\documentclass[12pt]{amsart}
\pdfoutput=1 
\allowdisplaybreaks 

\usepackage{hyperref}
\usepackage{amsmath,amsfonts,amssymb}

\newcommand\set[1]{\left\{#1\right\}}

\newtheorem{theorem}{Theorem}[section]
\newtheorem{corollary}[theorem]{Corollary}
\newtheorem{proposition}[theorem]{Proposition}
\newtheorem{lemma}[theorem]{Lemma}

\numberwithin{equation}{section}

\def\Hn{H_n}
\def\cP{\mathcal P}
\def\cG{\mathcal G}
\def\cS{\mathcal S}
\def\cF{\mathcal F}
\def\cV{\mathcal V}

\def\cE{\mathcal E}
\def\C{\mathbb C}
\def\R{\mathbb R}
\def\N{\mathbb N}

\def\l{\lambda}
\def\a{\alpha}

\def\vp{\varphi}
\def\eps{\varepsilon}

\def \Un {{\text{U}(n)}}  

\def \RE {\text{\rm Re}\,}
\def \IM {\text{\rm Im}\,}

\def\algop{{\mathbb D}_K}
\def\schK{\cS_K(\Hn)}

\def\gelfandgen#1{{\cG}^{#1}}

\begin{document}

\title[Gelfand pairs on the Heisenberg group and Schwartz functions]
{Gelfand pairs on the Heisenberg group and Schwartz functions}

\author[F. Astengo, B. Di Blasio, F. Ricci]
{Francesca Astengo, Bianca Di Blasio, Fulvio Ricci}

\address
{Dipartimento di Matematica\\
Via Dodecaneso 35\\
16146 Genova\\ Italy} \email{astengo@dima.unige.it}

\address
{Dipartimento di Matematica e Applicazioni\\
Via Cozzi 53\\
  20125 Milano\\ Italy}
\email{bianca.diblasio@unimib.it}

\address
{Scuola Normale Superiore\\
Piazza dei Cavalieri 7\\M
56126 Pisa\\ Italy}
\email{fricci@sns.it}

\thanks{Work partially supported
by  MIUR  and GNAMPA
}



\begin{abstract}
Let $\Hn$ be the $(2n+1)$--dimensional Heisenberg 
group and $K$ a compact group of automorphisms of $\Hn$
such that $(K\ltimes \Hn,K)$ is a Gelfand pair
. We prove that the Gelfand transform is a  topological isomorphism between the space of $K$--invariant Schwartz
 functions on $\Hn$ and the space of  
Schwartz function on a closed subset of $\R^s$
homeomorphic to the Gelfand spectrum of the Banach algebra
of $K$--invariant integrable functions on $\Hn$.
\end{abstract}

\maketitle

\section{Introduction}

A fundamental fact in harmonic analysis on $\R^n$
 is that the Fourier transform is a topological isomorphism
of the Schwartz space
 $\cS(\R^n)$ onto itself.

Various generalizations of this result for different classes of Lie groups exist in the literature, in particular in the context of Gelfand pairs, where the operator-valued Fourier transform can be replaced by the scalar-valued spherical transform. Most notable is the case of a symmetric pair of the noncompact type, with Harish-Chandra's definition of a bi-$K$--invariant Schwartz space on the isometry group (cf. \cite[p. 489]{He2}).

The definition of a Schwartz space on a Lie group becomes quite natural on a  nilpotent group $N$ (say connected and simply connected).
In that case one can define  the Schwartz space 
 by identifying $N$ with its Lie algebra via the 
exponential map.

The image of the Schwartz space on the Heisenberg group $\Hn$ under the group Fourier transform has been described by D. Geller \cite{G}. 
Let $K$ be a compact group of automorphisms of $\Hn$ such that convolution of  $K$--invariant functions is commutative, in other words assume that $(K\ltimes\Hn,K)$ is a Gelfand pair
. 
Then a scalar-valued spherical transform $\cG_K$ (where $\cG$ stands for ``Gelfand transform'') of $K$--invariant functions is available, 
and Geller's result can be translated into a characterization of the image under $\cG_K$ of
the space $\cS_K(\Hn)$ of $K$--invariant Schwartz functions. 
 In the same spirit, a 
characterization of $\cG_K\big(\cS_K(\Hn)\big)$ is given 
in~\cite{BJR} for   closed subgroups $K$ of the unitary group $\Un$.

In \cite{ADR} we have proved that, for $K$ equal to $\Un$ or $\mathbb T^n$ (i.e. for radial - resp. polyradial - functions), an 
analytically more significant description of $\cG_K\big(\cS_K(\Hn)\big)$ can be obtained by making use of natural homeomorphic embeddings of the Gelfand spectrum of $L^1_K(\Hn)$ in Euclidean space. The result is that $\cG_K\big(\cS_K(\Hn)\big)$ is the space of restrictions to the Gelfand spectrum of the Schwartz functions on the ambient space.
This condition of ``extendibility to a Schwartz function on the ambient space'' subsumes the rather technical condition on iterated differences in discrete parameters that are present in the previous characterizations.

In this article we extend the result of \cite{ADR} to general Gelfand pairs $(K\ltimes\Hn,K)$, with $K$ a compact group of automorphisms of $\Hn$. Some preliminary notions and facts are required before we can give a precise formulation of our main theorem.

Let $G$ be a connected Lie group and $K$
a compact subgroup thereof
such that $(G,K)$ is a Gelfand pair
and denote by 
$L^1(G//K)$ the convolution algebra of all bi-$K$--invariant 
integrable functions on $G$
.
The Gelfand spectrum of the 
commutative Banach algebra $L^1(G//K)$
may be identified with the set  of bounded 
spherical functions with the compact-open topology.  
Spherical 
functions are characterized as the joint eigenfunctions of all 
$G$--invariant differential operators on 
$G/K$, normalized in the $L^\infty$-norm. 
$G$--invariant differential operators on 
$G/K$ form a
commutative algebra $\mathbb D(G/K)$ which  is finitely generated~\cite{He}.

Given a finite set of generators 
$\set{V_1,\ldots,V_s}$
of $\mathbb D(G/K)$, we can  assign
 to each bounded spherical function $\phi$ the $s$-tuple 
 $\widehat V(\phi)=\left(\widehat{V_1}(\phi),\ldots,\widehat{V_s}(\phi)\right)$ of its eigenvalues
with respect to these generators. In this way, the Gelfand spectrum
 is identified with a closed subset $\Sigma_K^V$ of $\C^s$. When all bounded spherical functions are of positive type and the operators $V_j$ self-adjoint, $\Sigma_K^V\subset\R^s$. 
As proved in~\cite{FR},  the Euclidean topology induced on $\Sigma_K^V$ coincides with the compact-open topology on the set of bounded spherical functions (see also~\cite{BJRW} for $G=K\ltimes\Hn$ and $K\subset \Un$). When the Gelfand spectrum is identified with $\Sigma_K^V$, the spherical transform will be denoted by~$\gelfandgen{V}_K$.

Let $K$ be a compact group of automorphisms of $\Hn$ such that $(K\ltimes \Hn,K)$ is a Gelfand pair
and denote by $\algop$
the commutative algebra  of left-invariant and $K$--invariant  differential operators
on $\Hn$. Let $V=\set{V_1,\ldots,V_s}$ be a set of formally 
self-adjoint 
generators of  $\algop$.
We denote by $\cS(\Sigma_K^V)$ the space of
restrictions to $\Sigma_K^V$ of Schwartz functions on $\R^s$, endowed with the
quotient topology of ${\cS}(\R^s)/\{f:f_{|_{\Sigma_K^V}}=0\}$.

Our main result is the following:

\begin{theorem}\label{main}
The map $\gelfandgen{V}_K$ is a topological isomorphism between ${\cS}_K(\Hn)$ and 
${\cS}(\Sigma_K^V )$.
\end{theorem}

\smallskip
As customary, for us $\Hn$ is understood as $\R\times\C^n$, with canonical coordinates of the first kind. It is well known that, under the
action 
$$
k\cdot(t,z)=(t,kz)\qquad
\forall k\in \Un, \  (t,z)\in \Hn,
$$
 $\Un$ is a maximal compact
connected  group of automorphisms of $\Hn$, and
that every compact connected group of automorphisms of $\Hn$ is conjugated to 
a subgroup of  $\Un$.
Therefore, if $K$ is a compact group of automorphisms of $\Hn$,
then its identity  component in $K$ is conjugated to 
a subgroup of  $\Un$.
For most of this article, we deal with the 
case of $K$ connected  and contained 
in  $\Un$, leaving 
the discussion of the general case to the last section.

In Section~\ref{mainresult} we
show that it suffices to prove  Theorem~\ref{main}
 for one 
particular set  of generators of $\algop$, and 
in Section~\ref{sceltagen} we choose a convenient set of generators.
From a homogeneous Hilbert basis of $K$--invariant  
polynomials on $\R^{2n}$ we derive by symmetrization $d$ differential  
operators $V_1,\dots,V_d$, invariant under $K$; to these we add the  
central operator $V_0=i^{-1}\partial_t$, obtaining in this way a  
generating system of $d+1$ homogeneous operators.
Here we benefit from the deep study
 of the algebraic properties of the multiplicity-free actions of subgroups of $\Un$, developed in~\cite{BJR90, BJR92, BJR, BJRW, BR98, BR00}. In particular, rationality of the ``generalized binomial coefficients'' is  a crucial point in our argument (see formula~\ref{estensioneCk} below). It must be noticed that the proof of rationality in \cite{BR00} is based on the actual classification of multiplicity-free actions. In this respect, our proof 
 depends on the actual classification of the groups $K$ giving rise to Gelfand pairs.

After these preliminaries, we split the proof of Theorem~\ref{main} into two parts. 

In the
first part we show  that, if $m$ is a Schwartz function on $\R^{d+1}$, its restriction to $\Sigma_K^V$ is the
 Gelfand transform of a  function $f$ in $\cS_K(\Hn)$ (see Theorem~\ref{main1} below). The argument is based on Hulanicki's theorem \cite{H}, stating that Schwartz functions on the real line operate on positive  Rockland operators
on  graded nilpotent Lie groups producing convolution operators with Schwartz kernels. We adapt the argument in~\cite{V} to obtain a multivariate extension of 
Hulanicki's theorem (see Theorem~\ref{Veneruso}
below).

In the second part  we prove that the Gelfand transform  $\gelfandgen{V}_K f$
of a function  $f$ in $\cS_K(\Hn)$
 can be extended to a Schwartz function on $\R^{d+1}$ (see Theorem~\ref{main2} below). 
The proof begins with an extension to the Schwartz space 
 of the Schwarz--Mather theorem \cite{Ma, S}  for $C^\infty$ $K$--invariant functions (see Theorem~\ref{Schwarz-Mather} below). This allows us to extend   to a Schwartz function on $\R^d$ the restriction of $\gelfandgen{V}_K f$
 to the  ``degenerate part'' $\Sigma_0$ of the Gelfand spectrum (that corresponding to the one-dimensional representations of $\Hn$, or equivalently corresponding to the eigenvalue
 $0$ for $V_0$)%
.
 Then we associate to $f$ a Schwartz jet on 
 $\Sigma_0$.
  As in \cite{ADR}, the key tool here is the existence of ``Taylor coefficients'' 
  at points of $\Sigma_0$, proved by Geller~\cite{G} (see Theorem~\ref{geller} below). 
The Whitney extension theorem
 (adapted to Schwartz jets in Proposition~\ref{whitney}) gives therefore a Schwartz extension to $\R^{d+1}$ of the jet associated to $f$. To conclude the proof, it remains to prove that if $f\in\cS_K(\Hn)$ and the associated jet on $\Sigma_0$ is trivial, then $\cG_K^Vf$ admits a Schwartz extension. This is done by adapting an explicit interpolation formula already used in \cite{ADR} (see Proposition~\ref{estensioneCk}).
For a nonconnected group  
$K$, we remark that, calling $K_0$ the connected component of the  
identity, one can view the $K$-Gelfand spectrum as the quotient of  
the $K_0$-Gelfand spectrum under the action of the finite group $F=K/ 
K_0$ (it is known that if $(K\ltimes \Hn,K)$  is a Gelfand pair, so is $ (K_0\ltimes \Hn,K_0)$, cf.\cite{BJLR}). Starting from an $F$--invariant generating  
system of $K_0$--invariant differential operators, the $K$-Gelfand  
spectrum is then conveniently embedded in a Euclidean space by means  
of the Hilbert map associated to the action of $F$ on the linear span  
of these generators.

The paper is structured as follows.
In Section~\ref{preliminari} we recall some facts
about Gelfand pairs $(K\ltimes \Hn,K)$ and the associated 
Gelfand transform. 
In Section~\ref{mainresult} we
show that it suffices to prove  Theorem~\ref{main}
 for one 
particular set  of generators of $\algop$.
In Section~\ref{sceltagen} we choose a convenient set of generators
in the case where $K$ is a connected closed subgroup of $\Un$.
 In Section~\ref{FunctionalCalculus}
we show that every  function in $\cS\bigl(\R^{d+1}\bigr)$ gives rise to the Gelfand transform of a 
function  in $\schK$  via
functional calculus.
In Section~\ref{Mather} we extend the Schwarz-Mather theorem \cite{Ma, S}   to Schwartz spaces.
Section~\ref{viceversa} is devoted to 
define a Schwartz extension on  $\R^{d+1}$ of the Gelfand 
transform   of a 
function  in $\schK$.
In Section~\ref{Kgenerico} we show that our result holds
for all compact groups of automorphisms of $\Hn$.
%


\section{Preliminaries}\label{preliminari}

For the content of this section we refer
to~\cite{BJR90, BJR92, BJRW, C, FR, Kac}.

\subsection{The Heisenberg group and its representations}

We denote by $\Hn$ the Heisenberg group, i.e.,
the real manifold $\R\times\C^n$ equipped with the group law
$$
(t,z)(u,w)=\bigl(t+u+\tfrac12\IM w\cdot \overline{z} ,z+w\bigr)
\qquad t,u\in \R,\quad
\forall z,w\in \C^n,
$$
where $w\cdot \overline{z}$ is a short-hand writing for $\sum_{j=1}^n w_j\, \overline{z_j}$.
It is easy to check that Lebesgue measure $dt\,dz$ is a Haar measure on $\Hn$.

Denote by  $Z_j$ and $\bar Z_j$ the complex left-invariant vector
fields
$$
Z_j=\partial_{z_j}-{\textstyle{i\over4}}\, \bar{z}_j\,\partial_t
\qquad
\bar{Z}_j=\partial_{\bar{z}_j}+{\textstyle{i\over4}}\, z_j\,\partial_t ,
$$
and set $T= \partial_t$.


\medskip

For $\lambda>0$, denote by ${\cF}_\lambda$ the Fock space consisting of
the entire functions $F$ on  $\C^n$ such that 
$$
\|F\|_{\mathcal F_\lambda}^2= \left({\lambda\over2\pi}\right)^n\,\int_{\C^n} |F(z)|^2\,   
\,e^{-\frac{\lambda}{2} |z|^2}\,dz <\infty\ ,
$$
equipped with the norm $\|\cdot\|_{\mathcal F_\lambda}$.
Then $H_n$ acts on $\cF_\lambda$ through the unitary representation $\pi_\l$ defined by
$$
[\pi_\lambda(t,z)F](w)=e^{i\lambda t}\, 
           e^{-\frac{\lambda}{2} w\cdot\bar z-\frac{\lambda}{4} |z|^2}\, F(w+z)
\qquad \forall (z,t)\in \Hn,\!\!\quad F\in  {\cF}_{\lambda},\!\!\quad w\in \C^n \ ,
$$
and through its contragredient
$\pi_{-\lambda}(t,z)=\pi_{\lambda}(-t, \bar z)$. These are the Bargmann representations of $H_n$.

The space  $\cP(\C^n)$ of polynomials on $\C^n$ is dense in ${\cF}_\lambda$ 
($\lambda>0$)
and an orthonormal basis of 
$\cP(\C^n)$ seen as a subspace of  ${\cF}_\lambda$
is given by the monomials
$$
p_{\lambda,{\bf d}}(w)={w^{{\bf d}}
\over ((2/\lambda)^{|{{\bf d}}|}{{\bf d}}!)^{1/2}} 
\qquad{{\bf d}}\in \N^n.
$$
 
 \medskip
 
 Besides the Bargmann representations, $H_n$ has the one-dimensional representations
 $$
\tau_w(t,z)=e^{i\RE z\cdot\bar w}
$$
with $w\in \C^n$. The $\pi_\lambda$ ($\lambda\ne0$) and the $\tau_w$ fill up the unitary dual of $H_n$.

\subsection{Gelfand pairs $(K\ltimes \Hn,K)$}
Let $K$ be a compact groups of automorphisms of $\Hn$ such that the convolution algebra 
$L^1_K(\Hn)$ of integrable $K$--invariant functions on $\Hn$
 is abelian, i.e., such that $(K\ltimes \Hn,K)$ is a Gelfand pair. 
 It is known~\cite{BJLR} that this property holds for $K$ 
 if and only if it holds for its connected identity component $K_0$. 
 On the other hand, every compact, connected group of automorphisms of $\Hn$ is conjugate, modulo an automorphism, to a subgroup of $\Un$, acting on $\Hn$ via 
 $$
k\cdot (t,z)=(t,kz) \qquad \forall (t,z)\in \Hn,\quad
k\in \Un.
$$

 For the remainder of this section, we assume that $K$ is connected, contained in $\Un$,
 and $(K\ltimes \Hn,K)$
is a Gelfand pair.



For the one-dimensional representation $\tau_w$, we have $\tau_w(t,k^{-1} z)=\tau_{kw}(t,z)$. Therefore, $\tau_w(f)$ is a $K$--invariant function of $w$ for every $f\in L^1_K(H_n)$.

As to the Bargmann representations, if $k\in \Un$, $\pi_{\pm\lambda}^k(t,z)=\pi_{\pm\lambda}(t,k  z)$ is equivalent to $\pi_{\pm\lambda}$ for every $\lambda>0$ and every choice of the $\pm$ sign. Precisely, we set
$$
\nu_+(k)F(z)=F(k^{-1}z)\ ,\qquad
\nu_-(k)F(z)=F(\bar k^{-1}z)\ ,
$$
each of the two actions being the contragredient of the other. We then have
$$
\pi_\l(kz,t)=\nu_+(k)\pi_\l(z,t)\nu_+(k)^{-1}\ ,\qquad \pi_{-\l}(kz,t)=\nu_-(k)\pi_{-\l}(z,t)\nu_-(k)^{-1}\ .
$$

By homogeneity, the decomposition of $\cF_\l$ into irreducible invariant subspaces under $\nu_+$ (resp. $\nu_-$) is independent of $\l$ and can be reduced to the decomposition of the dense subspace $\cP(\C^n)$ of polynomials.

It is known since \cite{C,BJR90} that $(K\ltimes \Hn,K)$ is a Gelfand pair if and only if
$\nu_+$ (equivalently $\nu_-$) decomposes into irreducibles without multiplicities (in other words, if and only if it is {\it multiplicity free}).
The subgroups of $\Un$ giving multiplicity free actions on $\cP(\C^n)$ have been classified by Ka\v{c}~\cite{Kac} 
and the resulting Gelfand pairs $(K\ltimes \Hn,K)$ are listed in~\cite[tables 1 and 2]{BR00}.

Under these assumptions, the space $\cP(\C^n)$ of polynomials on $\C^n$ 
decomposes into $\nu_+$-irreducible subspaces,
$$
\cP(\C^n)=\sum_{\alpha\in \Lambda}P_\alpha\ ,
$$
where 
$\Lambda$ is an infinite subset of the unitary dual $\widehat K$ of $K$ and $\alpha$ denotes the equivalence class of the action on $P_{\alpha}$. The irreducible $\nu_-$--invariant subspaces of $\cP(\C^n)$ are 
$$
P_{\alpha'}=\{\bar p(z)=\overline{p(\bar z)}:p\in P_\alpha\}\ ,
$$
with the action of $K$ on $P_{\alpha'}$ being equivalent to the contragredient $\alpha'$ of $\alpha$.

On the other hand, $\cP(\C^n)=\sum_{m\in \N}\cP_m(\C^n)$,
where $\cP_m(\C^n)$ is the space of homogeneous
polynomials of degree $m$.
Since $\nu_\pm$ preserves each $\cP_m(\C^n)$,
each $P_\alpha$ is contained in
$\cP_m(\C^n)$ for some $m$. We then say that $|\alpha|=m$, so that
\begin{equation*} 
\cP_m(\C)=\sum_{|\alpha|=m}P_{\alpha}=\sum_{|\alpha|=m}P_{\alpha'}.
\end{equation*}

%


As proved in \cite{BJR90}, all the bounded spherical functions are of positive type. Therefore
there are two families of spherical functions. Those of
the first family are
$$
\eta_{Kw}(t,z)=\int_Ke^{i{\rm Re}\langle z,kw\rangle}\,dk,
\qquad w\in \C^n,
$$
parametrized by $K$--orbits in $\C^n$ and associated with
the one-dimensional representations of the Heisenberg group.

The elements of the second family are parametrized by pairs $(\l,\alpha)\in \R^*\times\Lambda$. If $\l>0$ and $\{v_1^\l,\ldots,v_{\text{dim} (P_\alpha)}^\l\}$ is an orthonormal basis
of $P_\alpha$ in the norm of ${\cF}_\l$, we have the spherical function 
\begin{equation}\label{sferiche}
\phi_{\l,\alpha}(t,z)=\frac{1}{\text{dim} (P_\alpha)}
\sum_{j=1}^{\text{dim} (P_\alpha)}\langle \pi_\l (t,z)v_j^\l\, ,v_j^\l\rangle_{\cF_\l}\ .
\end{equation}

Taking, as we can, $v_j^\l$ as $\l^{|\alpha|/2}v_j^1$, we find that
$$
\phi_{\l,\alpha}(z,t)=\phi_{1,\alpha}(\sqrt\l z,\l t)\ .
$$

For $\l<0$, the analogous matrix entries of the contragredient representation give the spherical functions
$$
\phi_{\l,\alpha}(z,t)=\overline{\phi_{-\l,\alpha}(z, t)}\ .
$$

Setting for simplicity $\phi_\alpha=\phi_{1,\alpha}$, 
\begin{equation*} 
\phi_{\alpha}(z,t)
=e^{it}\,q_\alpha(z,\bar z)\,e^{-|z|^2/4},
\end{equation*} 
where $q_\alpha\in \cP(\C^n)\otimes\overline{\cP(\C^n)}$ is a real $K$--invariant polynomial
 of degree $2|\alpha|$ in $z$ and $\bar z$ (cf.~\cite{BJR92}).

Denote by $\algop$ the algebra of left-invariant and
$K$--invariant differential operators on $\Hn$. The symmetrization map establishes a linear bijection from the space of $K$--invariant elements in the symmetric algebra over $\frak h_n$ to $\algop$. Therefore every element $D\in\algop$ can be expressed as
 $\sum_{j=0}^m D_jT^j$, where $D_j$ is the symmetrization of a $K$--invariant polynomial
  in $Z,\bar Z$.

Let $D$ be the symmetrization of the $K$--invariant polynomial $P(Z,\bar Z,T)$. With  $\widehat D(\phi)$ denoting the eigenvalue of $D\in\algop$
on the  spherical function $\phi$, we have
\begin{equation}\label{autoveta}
\widehat {D}(\eta_{Kw})=P(w,\bar w,0)
\end{equation}
for the spherical functions associated to the one-dimensional representations

For $\l\ne0$, $d\pi_\lambda(D)$
commutes with the action of $K$ and therefore it preserves  
each $P_\alpha$. By 
Schur's lemma,
$d\pi_\lambda(D)_{|_{P_\alpha}}$ is a scalar operator $c_{\lambda,\alpha}(D)\,I_{{P_\alpha}}$. It follows from (\ref{sferiche}) that 
$$
\widehat {D}(\phi_{\lambda,\alpha})=c_{\lambda,\alpha}(D)\ .
$$

In particular, for $D=T$, we have
$$
\widehat T(\eta_{Kw})=0\ ,\qquad \widehat T\big(\phi_{\l,\alpha}\big)=i\l \ .
$$
%

\section{Embeddings of the Gelfand spectrum}\label{mainresult}


%

Let $K$ be a compact group of automorphisms of $\Hn$
such that  $(K\ltimes \Hn,K)$
is a Gelfand pair.
The Gelfand spectrum of 
the commutative Banach algebra $L^1_K(\Hn)$ is  the set of
bounded  spherical functions endowed with the compact-open topology. 
Given a set
 $V=\{V_0,\,\,V_1,\ldots,V_d\}$  of formally self-adjoint generators of $\algop$,  we assign
 to each spherical function $\phi$ the $(d+1)$-tuple 
 $\widehat V(\phi)=\left(\widehat{V_0}(\phi),\,\,\widehat{V_1}(\phi),\ldots,\widehat{V_d}(\phi)\right)$. 
 Since $d\pi(V_j)$ is formally self-adjoint for every irreducible
 representation $\pi$ , $\widehat V(\phi)$  is in~$\R^{d+1}$.
It has been proved,  in a more general
context~\cite{FR}, that $\Sigma_K^V=\{\widehat V(\phi):\phi\text{ spherical}\}$ is closed in $\R^{d+1}$ and homeomorphic to the Gelfand spectrum via
$\widehat V$ (see also~\cite{BJRW}).
Once we have identified the Gelfand spectrum with $\Sigma_K^V$,
the Gelfand transform of a function $f$ in $L^1_K(\Hn)$
can be defined on the closed subset $\Sigma_K^V$ of $\R^{d+1}$ as
$$
(\gelfandgen{V}_K f)\,(\widehat V(\phi))=\int_{\Hn} f\,
\phi.
$$

In order to prove Theorem~\ref{main}, we first show that different choices of the generating system $V$ give rise to natural isomorphisms among the corresponding restricted Schwartz spaces $\cS(\Sigma_K^V)$. It will then suffice to prove  Theorem~\ref{main} for one 
particular set  of generators.

On the Schwartz space ${\cS}(\R^m)$ we consider  the following family of
norms, parametrized  by a nonnegative integer  $p$:
$$
\|f\|_{(p,\R^m)}=\sup_{y\in \R^m,|\alpha|\leq p}
(1+|y|)^p |\partial^\alpha f(y)|.
$$

\begin{lemma}\label{mappapropria2}
Let $E$ and $F$ be closed subsets of
$\R^n$ and $\R^m$ respectively. 
Let $P:\R^n\to\R^m$ and $Q:\R^m\to\R^n$ be  polynomial maps such that  $P(E)=F$ and $Q\circ P$ is the identity on~$E$.
Given $f$ in $\cS(F)$ we let $P^\flat f=f\circ P|_E$.
Then $P^\flat$ maps $\cS(F)$ in $\cS(E)$ continuously.
\end{lemma}

\begin{proof}
We show that if $f$ is in $\cS(\R^m)$, then
$P^\flat f$ can be extended to a function $\widetilde{P^\flat f}$
in $\cS(\R^n)$ in a linear and continuous way. Let $\Psi$ be a smooth function on $\R^n$ 
 such that
$\Psi(t)=1$ if $|t|\leq 1$ and  $\Psi(t)=0$ if $|t|> 2$.
Define 
$$
\widetilde{P^\flat f}(x)=\Psi(x-Q\circ P(x))\,\,(f\circ P)(x) 
\qquad \forall x\in \R^n.
$$
Clearly $\widetilde{P^\flat f}$ is smooth and 
$\widetilde{P^\flat f}|_E=P^\flat f$.
Moreover $\widetilde{P^\flat f}$ is zero when $|x-Q\circ P(x)|>2$,
so it suffices to prove rapid decay for $|x-Q\circ P(x)|\leq 2$.
Note that
there exists $\ell$ in $\N$ such that
$$
|x|\leq 2+|Q( P(x))|\leq C\,(1+|P(x)|)^\ell
\qquad \forall x\in \R^n,\quad |x-Q\circ P(x)|\leq 2.
$$
Therefore given a positive integer $p$ there exists
a positive integer $q$ such that 
$
\|\widetilde{P^\flat f}\|_{(p,\R^n)}\leq C\, 
  \|f\|_{(q,\R^m)}.
$
The thesis follows immediately from  the
definition of the quotient topology on
$\cS(F)$ and $\cS(E)$.
%
\end{proof}

\begin{corollary}\label{cambiogen}
Suppose that $\{V_0,\ldots,V_d\}$ and $\{W_0,\ldots,W_s\}$ are two
sets of formally
self-adjoint generators of $\algop$. Then the spaces 
${\cS}(\Sigma_K^V)$ and ${\cS}(\Sigma_K^W)$ are topologically isomorphic.
\end{corollary}

\begin{proof}
There exist real polynomials
$p_j$, $j=0,1,\ldots,s$, and 
$q_h$, $h=0,1,\ldots,d$, such that  
$
W_j=p_j(V_0,\ldots,V_d)
$
and
$
V_h=q_h(W_0,\ldots,W_s).
$

Setting $P=(p_0,\,p_1,\ldots,p_s):\R^{d+1} \rightarrow\R^{s+1}$  
and $Q=(q_0,\,q_1,\ldots,q_d):\R^{s+1} \rightarrow\R^{d+1} $,  
we can apply Lemma~\ref{mappapropria2} in both directions.
\end{proof}

\section{Choice of the generators} \label{sceltagen}

In this section $K$ shall 
 be a closed  connected subgroup of $\Un$
 such that  $(K\ltimes \Hn,K)$
is a Gelfand pair.
 The subject of the following lemma
 is the choice of a convenient
 set  of formally
self-adjoint generators of $\algop$.

\begin{lemma}\label{autinteri}
A generating system
$\{V_0=-iT,V_1,\ldots,V_d\}$ can be chosen such that, for each $j=1,\ldots,d$,
 \begin{enumerate}
\item\label{uno}
$V_j$ is homogeneous of even order $2\,m_j$;

\item\label{due}
$V_j$ is formally
self-adjoint and $\widehat V_j(\phi_{\alpha})$ is a positive integer
for every  $\a$  in $\Lambda$;

\item\label{tre}
$\widehat V_j(\eta_{Kw})=\rho_j (w,\bar w)$, for every $w$ in $\C^n$,
where  $\rho_j$ is a
nonnegative homogenous polynomial  of degree $2\,m_j$,
strictly positive outside of the origin.
\end{enumerate}
\end{lemma}

Notice that~(\ref{uno}) and (\ref{due}) imply that when $j=1,\ldots,d$
\begin{equation}\label{scaling}
\widehat V_j(\phi_{\l,\a})=|\l|^{m_j}\,\widehat V_j(\phi_{\a}),
\qquad
\forall\l\in\R\setminus\{0\},\quad\forall\a\in\Lambda.
\end{equation}

\begin{proof}
Let $\C^n_\R$ denote $\C^n$ with the underlying structure of a real vector space. We denote by $\cP(\C^n_\R)\cong \cP(\C^n)\otimes\overline{\cP(\C^n)}$ the algebra of polynomials in $z$ and $\bar z$, and by $\cP^K(\C^n_\R)$ the subalgebra of $K$--invariant polynomials.

The fact that the representation
of $K$
on $\cP(\C^n)$ is multiplicity free implies that the trivial representation is contained in $P_\alpha\otimes\overline{P_\beta}\subset \cP(\C^n_\R)$ if and only if $\alpha=\beta$, and with multiplicity one in each of them. Therefore a linear basis of $\cP^K(\C^n_\R)$ is given by the polynomials
\begin{equation}\label{invarianti}
p_\alpha(z,\overline z)=
\sum_{h=1}^{{\rm dim}(P_\a)}v_h(z)\,\overline{v_h(z)}=
\sum_{h=1}^{{\rm dim}(P_\a)}|v_h(z)|^2\ ,
\end{equation}
where  $\{v_1,\ldots, v_{{\rm dim} (P_\alpha)}\}$ is any orthonormal
 basis of $P_\alpha$ in the ${\cF}_1$-norm.

 A result in~\cite{HU} ensures that there exist $\delta_1,\ldots,\delta_d$ in $\Lambda$
 such that the polynomials 
 $$
\gamma_j=p_{\delta_j}\qquad j=1,\ldots,d,
 $$
freely generate  $ \cP^K(\C_\R^n)$.
In~\cite{BR00} the authors prove that 
$\gamma_1,\ldots,\gamma_d$ have
rational coefficients.
More precisely, setting $m_j=|\delta_j|$, each $\gamma_j$ can be written in the form
$$
\gamma_j(z,\bar z)=
 \sum_{|{\bf a}|=|{\bf b}|=m_j}
 \theta^{(j)}_{{\bf a},{\bf b}}\,\,z^{\bf a}\bar z^{\bf b},
 $$
 where ${\bf a},{\bf b}$ are in $\N^n$ 
 and  $\theta^{(j)}_{{\bf a},{\bf b}}$ are rational numbers.
 
The symmetrization $L_{\gamma_j}$ of $\gamma_j(Z,\bar Z)$ is a homogenous operator
of degree $2m_j$  in $\algop$ with rational coefficients, and $\{-iT,\,L_{\gamma_1},\ldots,L_{\gamma_d}\}$ generate $\algop$
~\cite{BJR92, BJRW}. Moreover, the eigenvalues
$\widehat L_{\gamma_j}(\phi_{\alpha})$ are rational numbers~\cite{BR98,BR00}.

Fix any positive integer $m$ and  denote by $M_{j,m}$ the matrix which 
 represents the restriction of 
 $ d\pi_1\left(L_{\gamma_j}\right)$
to $\cP_m(\C^n)$ in the basis of
 monomials $w^\a$, $|\a|=m$. 
For every $F$ in $\cF_1(\C^n)$,
\begin{equation*}
[d\pi_1(Z_h)F](w)=[\partial_{w_h}F](w) \qquad
[d\pi_1(\bar Z_h)F](w)=-{1\over 2}\,w_h\,F(w)
\qquad \forall w\in \C^n,
\end{equation*}
for $h=1,\ldots,n$.
Therefore $M_{j,m}$ has rational entries,
with denominators varying in a finite set independent
of  $m$. We can then take $N$ such that the matrices $NM_{j,m}$
have integral entries, for all $m$ and $j=1,\ldots, d$.
Thus the characteristic polynomial of $NM_{j,m}$ is monic,
with integral coefficients and 
rational zeroes
$N\widehat {L_{\gamma_j}}(\phi_{\alpha})$; therefore these zeroes must be integers. 
They all have the same sign, independently of $m$, equal to $(-1)^{m_j}$ (cf.~\cite{BJRW}).

We then define
 $$
 V_j=N\,(-1)^{m_j}\,L_{\gamma_j}+\mathcal L^{m_j},\qquad j=1,\ldots,d,
 $$
 where $\mathcal L=-2\sum_{j=1}^n
 \bigl( Z_j\overline Z_j+\overline Z_j Z_j\bigr)$ is the $\Un$--invariant sublaplacian, satisfying $\widehat{\mathcal L}(\phi_\alpha)=2|\alpha|+n$. We show that
$\{V_0=-iT,V_1,\ldots,V_d\}$ is a set of generators 
satisfying the
required conditions. 

Since $\sum_{m_j=1}\gamma_j=\frac12|z|^2$ (cf. \cite{BJRW}), we have
$$
\sum_{m_j=1} L_{\gamma_j}= \sum_{j=1}^n
 \bigl( Z_j\overline Z_j+\overline Z_j Z_j\bigr)=-\frac12 {\mathcal L}\ ,
$$
and then
\begin{equation}\label{scomposL}
\sum_{m_j=1}V_j
=
\sum_{m_j=1}\left(-
N\,L_{\gamma_j}+\mathcal L\right)
=\left(\frac{N}2+r\right)\,\mathcal L\ ,
\end{equation}
where  $r$ is  the cardinality of the set $\set{\delta_j\,:\, m_j=|\delta_j|=1}$.
Therefore each $L_{\gamma_j}$ is a polynomial in 
$V_1,\ldots,V_d$.
Since $-iT,L_{\gamma_1},\ldots,L_{\gamma_d}$
generate $\algop$, the same holds for $-iT,V_1,\ldots,V_d$.

Condition~(\ref{uno}) follows from the homogeneity 
of $L_{\gamma_j}$ and $\mathcal L^{m_j}$, and from (\ref{invarianti}). 
Since the polynomials in (\ref{invarianti}) are real-valued, then
the $V_j's$ are formally self-adjoint and 
condition~(\ref{due}) is easily verified.
Finally condition~(\ref{tre}) follows from (\ref{autoveta}), which gives
$$
\widehat{L_{\gamma_j}}(\eta_{Kw})=(-1)^{m_j}\,\gamma_j(w,\bar w)\ ,
\qquad
\widehat{{\mathcal L}^{m_j}}(\eta_{Kw})=|w|^{2m_j}
$$
for all $w$ in $\C^n$ and $j=1,\ldots,d$. 
\end{proof}

\bigskip

Let  $V=\{V_0,V_1,\ldots,V_d\}$ denote the privileged  set of generators
chosen in Lemma~\ref{autinteri}.
 We set $\rho=(\rho_1,\ldots,\rho_d)$ the polynomial map in (\ref{tre}) of Lemma~\ref{autinteri}.

Coordinates in $\R^{d+1}$ will be denoted by $(\lambda,\xi)$, 
with $\lambda$ in $\R$ and $\xi$ in $\R^{d}$.
So, if $(\lambda,\xi)=\widehat {V}(\phi)$, then
either $\lambda=-i\widehat T(\phi)=0$, in which case $\phi=\eta_{Kw}$ and $\xi=\rho(w,\bar w)$,
or $\lambda=-i\widehat T(\phi)\neq 0$, in which case $\phi=\phi_{\l,\alpha}$ and 
$\xi_j=\widehat{V_j}(\phi_{\l,\alpha})=|\l|^{m_j}\, \widehat{V_j}(\phi_{\alpha})$.

The spectrum $\Sigma_K^V$ consists therefore of two parts. The first part is $\Sigma_0=\{0\}\times \rho(\C^n)$, a semi-algebraic set. The second part, $\Sigma'$, is the countable union of the curves $\Gamma_\alpha(\l)=\widehat V(\phi_{\l,\alpha})$, $\l\ne0$. Each $\Gamma_\alpha$ and $\Sigma_0$ are homogeneous with respect to the dilations 
\begin{equation}\label{dilations}
(\l,\xi_1.\dots,\xi_d)\mapsto (t\l,t^{m_1}\xi_1.\dots,t^{m_d}\xi_d)\ ,\qquad (t>0) 
\end{equation}
and to the symmetry $(\l,\xi_1.\dots,\xi_d)\mapsto(-\l,\xi_1.\dots,\xi_d)$. By our choice of $V$, $\Sigma'\cap\{\l=1\}$ is contained in the positive integer lattice. Moreover
$\Sigma'$ is dense in $\Sigma_K^V$ (cf. \cite{BJRW, G})

For the sake of brevity, we denote by $\hat f$ the Gelfand transform $\gelfandgen{V}_K f$ of the integrable $K$--invariant function  $f$ .

\section{Functional calculus}\label{FunctionalCalculus}

In this section we prove one of the two implications of Theorem~\ref{main} for  a  closed connected  subgroup $K$ of $\Un$
such that  $(K\ltimes \Hn,K)$
is a Gelfand pair. More precisely we prove 
that if $m$ is a Schwartz function on $\R^{d+1}$, its restriction to $\Sigma_K^V$ is the
 Gelfand transform of a  function $f$ in $\cS_K(\Hn)$ (see Theorem~\ref{main1} below).

The proof is based
on a result of Hulanicki~\cite{H} (see Theorem~\ref{hula} below) on functional calculus for Rockland operators on graded groups and a multi-variate extention of it.

A Rockland operator
$D$ on a graded Lie group $N$ is
a formally self-adjoint left-invariant differential operator on $N$ which is homogeneous
with respect to the dilations and such that, for every nontrivial 
 irreducible representation
$\pi$
of $N$, the operator $d\pi(D)$ is injective on the space of $C^\infty$
vectors.

As noted in~\cite{HJL}, it follows from \cite{NS} and \cite{HN} that
 a Rockland operator $D$ is essentially self-adjoint on the Schwartz space $\cS(N)$, as well as $d\pi(D)$ on the G\aa rding space for every unitary representation $\pi$. We keep the same symbols for the self-adjoint extensions of such operators.

Let $N$  be a graded Lie group,
and let $|\cdot|$ be a homogeneous gauge on it.
 We say that a function on $N$ is Schwartz if and only if it is represented by a Schwartz function on the Lie algebra $\frak n$ in any given set of canonical coordinates. The fact that changes of canonical coordinates are expressed by polynomials makes this condition independent of the choice of the coordinates.
Given a homogeneous basis  $\{X_1,\dots,X_n\}$ 
 of the Lie algebra $\frak n$ we keep the same notation 
 $X_j$ for the associated  left-invariant vector
fields on $N$.
Following~\cite{FS}, we shall consider the following family of
norms on ${\cS}(N)$, parametrized  by a nonnegative integer $p$:
\begin{equation}\label{normeSchwartz}
\|f\|_{(p,N)}=\sup\{(1+|x|)^p\, |X^I f(x)|
\,:x\in N\,,\ \text{deg}\, X^I\leq p\},
\end{equation}
where $X^I=X_1^{i_1}\cdots X_n^{i_n}$
and $\text{deg}\,X^I=\sum\,i_j\text{deg}\,X_j$.
Note that the Frech\'et space structure induced on $\cS(N)$ by this family of norms is independent of the choice of the $X_j$ and is equivalent to that induced from $\cS(\frak n)$ via composition with the exponential map.

\begin{theorem}\label{hula}
\cite{H}
Let $D$ a positive Rockland operator on a graded
Lie group $N$ and let $D=\int_0^{+\infty}\l\, dE(\l)$
be its spectral decomposition. If $m$ is in $\cS(\R)$ and
$$
m(D)=\int_0^{+\infty}m(\l) \,dE(\l),
$$
then there exists $M$ in $\cS(N)$ such that
$$
m(D)f=f\ast M\qquad \forall f \in\cS(N).
$$
Moreover for every $p$ there exists 
$q$ such that
$$
\|M\|_{(p,N)}\leq C\, \|m\|_{(q,\R)}.
$$
\end{theorem}

Suppose that $D_1,\ldots,D_s$ form a commutative
family of  self-adjoint operators 
on $N$ (in the sense that they have commuting spectral resolutions). Then they admit a joint spectral 
resolution and one can define the bounded operator
$m(D_1,\ldots,D_s)$ for any bounded Borel function $m$ on their joint spectrum in $\R^s$.
The following theorem was proved in \cite{V} in a special situation.

\begin{theorem}\label{Veneruso}
Suppose that 
$N$  is a graded Lie group and
$D_1,\ldots,D_s$ form a commutative
family of positive
Rockland operators on $N$. If $m$ is in $ \cS(\R^s)$, then
 there exists $M$ in $\cS(N)$ such that
$$
m(D_1,\ldots,D_s)f=f*M\ .
$$
Moreover, for every $p$ there exists
$q$ such that
$$
\|M\|_{(p,N)}\leq C\, \|m\|_{(q,\R^{d})}
$$
\end{theorem}

\begin{proof}
We prove the theorem by induction on $s$.
By Theorem~\ref{hula} the thesis holds
 when $s=1$. Let $s\geq 2$ and
suppose that the thesis holds for $s-1$.
Let $m(\l_1,\dots,\l_s)$ be in $ \cS(\R^s)$. Then
there exist sequences   $\{\psi_k\}$
in $ \cS(\R^{s-1})$ and $\{\varphi_k\}$
in $ \cS(\R)$ such that
$
m(\l_1,\dots,\l_{s-1},\l_s)
=\sum_k\psi_k(\l_1,\dots,\l_{s-1}) \varphi_k(\l_s)
$
and $\sum_k \left\|\psi_{k}\otimes \varphi_{k}\right\|_{N,\R^s}<\infty$
for every~$N$. 

For a proof of this, one can first decompose $m$ as a sum of $C^\infty$-functions supported in a sequence of increasing balls and with rapidly decaying Schwartz norms, and then separate variables in each of them by a Fourier series expansion (see also~\cite{V}).

By the inductive hypothesis, for every $k$ there exist $\Psi_{k}$  and $\Phi_{k}$
in $\cS(N)$  such that
$\psi_{k}(D_1,\ldots, D_{s-1})f=f\ast\Psi_{k}$ and
$\varphi_{k}(D_s)f=f\ast \Phi_{k}$ for every $f$ in $\cS(N)$.
Then
$$
\begin{aligned}
\psi_{k}\otimes \varphi_{k}
\left(D_1,\ldots, D_{s-1},D_s\right)f&=
\psi_{k}\left(D_1,\ldots, D_{s-1}\right)\, \varphi_{k}
\left(D_s\right)f
\\
&=f\ast\Phi_{k}\ast\Psi_{k}\ .
\end{aligned}
$$

By straightforward computations (cf. \cite[Proposition 1.47]{FS}) and the inductive hypothesis, we obtain that
$$
\begin{aligned}
\|\Psi_{k}\ast\Phi_{k}\|_{(p,N)}
&\leq
C_p\,\|\Psi_{k}\|_{(p',N)}\,\|\Phi_{k}\|_{(p''+1,N)}\\
&\leq C_p\,\|\psi_k\|_{(q,\R^{s-1})}
\|\varphi_k\|_{(q',\R)}\\
&\leq C_p\,\|\psi_k\otimes\varphi_k\|_{(q'',\R^s)}\ .
\end{aligned}
$$

Since the series $\sum_k\psi_k\otimes \varphi_k$ 
is  totally convergent in every Schwartz norm
on $\R^s$,
 there exists a function $F$ in $\cS(N)$
such that
$$
\sum_k \Psi_{k}\ast\Phi_{k}=F
$$
and hence
$$
\sum_k f\ast\Psi_{k}\ast\Phi_{k}= f\ast F
$$
for every $f$ in $\cS(N)$.

On the other hand, if $M\in\cS'(N)$ is the convolution kernel of $m\left(D_1,\ldots, D_{s-1},D_s\right)$, 
by the Spectral Theorem,
$$
\begin{aligned}
m\left(D_1,\ldots, D_{d-1},D_d\right)\,f
&=
\sum_k\psi_k\left(D_1,\ldots, D_{d-1}\right)\otimes
\varphi_k\left(D_d\right)\,f
\cr
&=
\sum_kf\ast\Psi_{k}\ast\Phi_{k},
\end{aligned}
$$
 for every $f$ in $\cS(N)$,
with convergence in $L^2(N)$. Therefore $\sum_k \Psi_{k}\ast\Phi_{k}$ converges to $M$ in $\cS'(N)$, i.e. $F=M$ and
$M$ is in $\cS(N)$.

Finally, given $p$ there exists $q$ such that
$$
\| M\|_{(p,N)}
\leq
\|m\|_{(q,\R^d)}\ .
$$

This follows from the Closed Graph Theorem. Indeed, we have shown that
there is a linear  correspondence $m\mapsto M$  from $\cS(\R^{d})$ to $\cS(N)$;
moreover, reasoning as before, if $m_h\to m$ in $\cS(\R^{d})$ and $M_h\to \varphi$
in $\cS(N)$ then $M_h\to M$  by the Spectral Theorem in $\cS'(N)$. Therefore $\varphi=M$.
\end{proof}

Going back to our case, we prove that the operators $V_1,\dots,V_d$ satisfy the hypotheses of Theorem \ref{Veneruso}.

\begin{lemma}\label{verificaRockland}
The differential operators 
$V_1,\ldots, V_d$ defined in Lemma~\ref{autinteri}
form a commutative family of positive Rockland operators on $\Hn$.
\end{lemma}

\begin{proof} Suppose that $j=1,\ldots,d$.
The  formal self-adjointness of $V_j$ follows 
directly from  its definition and the identity
$\int_{\Hn} Z_k f\, \overline g
=-\int_{\Hn} f\,  Z_k \overline g$, for every pair of Schwartz functions $f$ and $g$ on $\Hn$.

We check now the injectivity condition on the image of $V_j$ in the nontrivial irreducible unitary representations of $H_n$.

By (\ref{tre}) in Lemma~\ref{autinteri}, $d\tau_w(V_j)=\rho_j(w)>0$ for $w\neq0$.

As to the Bargmann representations $\pi_\l$, the G\aa rding space in $\cF_{|\l|}$ can be characterized
as the space of those $F=\sum_{\alpha\in\Lambda}F_\alpha$ (with $F_\alpha$ in $P_\alpha$ for $\l>0$ and in $P_{\alpha'}$ for $\l<0$) such that
$$
\sum_{\alpha\in\Lambda}(1+|\alpha|)^N\|F_\alpha\|_{\cF_{|\l|}}^2<\infty
$$
for every integer $N$. For such an $F$,
$$
V_jF=\sum_{\alpha\in\Lambda}\widehat{V_j}(\phi_{\l,\alpha})F_\alpha
$$
where the series is convergent in norm, and therefore it is zero if and only if $F_\alpha=0$ for every $\alpha$.

Positivity of $V_j$ follows from Plancherel's formula: for $f\in\cS(H_n)$,
$$
\int_{\Hn} V_jf\, \bar f=
\left(\frac1{2\pi}\right)^{n+1}
\int_0^{+\infty}\sum_{\alpha\in\Lambda}\widehat{V_j}(\phi_{\l,\alpha})\big(\|\pi_\l(f)_{|_{P_\alpha}}\|_{HS}^2+\|\pi_{-\l}(f)_{|_{P_{\alpha'}}}\|_{HS}^2\big)\,\l^n\,d\l\ge0\ ,
$$
since the eigenvalues  $\widehat{V_j}(\phi_{\lambda,\a})$ are  positive.

Given a Borel subset $\omega$ of $\R^+$, define the operator $E_j(\omega)$ on $L^2(H_n)$ by
\begin{equation}\label{resolution}
\pi_\l\big(E_j(\omega)f\big)=\sum_{\alpha\in\Lambda}
\chi_\omega\big(\widehat{V_j}(\phi_{\l,\alpha})\big))\pi_\l(f)\Pi_{\l,\alpha}\ ,
\end{equation}
where $\chi_\omega$ is the characteristic function of 
$\omega$ and  $\Pi_{\l,\alpha}$
is the orthogonal projection
of $\cF_{|\l|}$ onto $P_\alpha$ if $\l>0$, or onto $P_{\alpha'}$ if $\l<0$. Then $E_j=\{E_j(\omega)\}$ defines, for each $j$, a resolution of the identity, and, for $f\in\cS(H_n)$,
$$
\int_0^{+\infty}\xi\,dE_j(\xi)f=V_jf\ .
$$

Therefore $E_j$ is the spectral resolution of the self-adjoint extension of $V_j$.
It is then clear that $E_j(\omega)$ and $E_k(\omega')$ commute for every $\omega,\omega'$ and $j,k$.
\end{proof}

\begin{corollary}\label{corveneruso}
Let $V_0,\ldots, V_d$ be the differential operators
defined in Lemma~\ref{autinteri}. If $m$ is in $ \cS(\R^{d+1})$, then
 there exists $M$ in $\schK$ such that
$$
m(V_0,\ldots, V_d)f=f\ast M\qquad \forall f \in\cS(\Hn).
$$
Moreover, for every $p$ there exists
$q$ such that
$$
\|M\|_{(p,\Hn)}\leq C\, \|m\|_{(q,\R^{d+1})}
$$
\end{corollary}

\begin{proof}  
We replace $V_0=-iT$ by $\tilde V_0=-iT+2\mathcal L$. By (\ref{scomposL}),
$\tilde V_0$ is a linear combination of the $V_j$. Therefore $m(V_0,\ldots, V_d)=\tilde m(\tilde V_0,\ldots, V_d)$, where $\tilde m$ is the composition of $m$ with a linear transformation of $\R^{d+1}$.

Moreover, $\tilde V_0$ is a positive Rockland operator (cf. \cite{FSarticolo}), which commutes with the other $V_j$ because so do $V_0$ and $\mathcal L$.
Applying Theorem \ref{Veneruso},
$$
\tilde m(\tilde V_0,\ldots, V_d)f=f\ast M\ ,
$$
with $M\in\cS_K(H_n)$ and 
$$
\|M\|_{(p,\Hn)}\leq C\, \|\tilde m\|_{(q,\R^{d+1})}\leq C'\, \|m\|_{(q,\R^{d+1})}\ .
$$
\end{proof}
%


\begin{theorem}\label{main1}
Suppose that $m$ is a Schwartz function on $\R^{d+1}$. Then there exists
a function 
$M$ in $\schK$ such that $\widehat M=m_{|_{\Sigma_K^V}}$. Moreover the map
$m\longmapsto M$ is a continuous linear 
operator from $\cS(\R^{d+1})$ to $\schK$.
\end{theorem}

\begin{proof} 
It follows from (\ref{resolution})
 that the joint spectrum of $V_0,\dots,V_d$ is $\Sigma_K^V$ (cf. \cite[Theorem 1.7.10]{GV}). Therefore the continuous map $m\mapsto M$ of Corollary \ref{corveneruso} passes to the quotient modulo $\{m:m=0 \text{ on }\Sigma_K^V\}$.

On the other hand, by (\ref{resolution}),
$$
\pi_\l\big( m(V_0,\ldots, V_d)f\big)=\sum_{\alpha\in\Lambda}
m\big(\widehat{V}(\phi_{\l,\alpha})\big)\pi_\l(f)\text{proj}_{\l,\alpha}\ ,
$$
and this must coincide with $\pi_\l(f)\pi_\l(M)$. It follows that
$\widehat M\big(\widehat{V}(\phi_{\l,\alpha})\big)=m\big(\widehat{V}(\phi_{\l,\alpha})\big)$ for every $\l,\alpha$. By density, $\widehat M=m_{|_{\Sigma_K^V}}$.
\end{proof}

\section{Extension of Schwartz invariant functions on $\R^m$}
\label{Mather}

Suppose that $K$ is a compact Lie group acting 
orthogonally on $\R^m$.
 It follows from Hilbert's Basis Theorem \cite{W} that
the algebra of $K$--invariant polynomials on $\R^m$ is
finitely generated. Let $\rho_1,\ldots,\rho_d$
be a set of generators
and denote by $\rho=(\rho_1,\ldots,\rho_d)$
the corresponding map from $\R^m$ to $\R^d$. The image $\Sigma=\rho(\R^m)$ of $\rho$ is closed in $\R^d$.

If $h$
is a smooth function on $\R^d$, then $h\circ \rho$
is in $C^\infty_K(\R^m)$, the space of $K$--invariant smooth
functions on $\R^m$.
G. Schwarz \cite{S} proved that
the map $h\mapsto h\circ\rho$ is  surjective
from $C^\infty(\R^d)$ to $C^\infty_K(\R^m)$, so that, passing to the quotient, it establishes an isomorphism between $C^\infty(\Sigma)$ and $C^\infty_K(\R^m)$.

J. Mather~\cite{Ma} proved that the map $h\mapsto h\circ\rho$ is split-surjective, i.e. there is a continuous linear operator $\mathcal E:C^\infty_K(\R^m)\rightarrow C^\infty(\R^d)$ such that $(\mathcal E f)\circ\rho=f$ for every $f\in C^\infty_K(\R^m)$.

From this one can derive the following analogue of the Schwarz--Mather theorem for 
$\cS_K(\R^m)$.

 \begin{theorem}\label{Schwarz-Mather}
 There is a continuous linear operator $\mathcal E': \cS_K(\R^m)\rightarrow \cS(\R^d)$ such that $(\mathcal E' g)\circ\rho=g$ for every $g\in \cS_K(\R^m)$.
 In particular, the map
$g\mapsto g\circ \rho $ is an isomorphism between 
$  \cS(\Sigma)$ and $  \cS_K(\R^{m})$. 
\end{theorem}

\begin{proof} It follows from Lemma \ref{mappapropria2} that the validity of the statement is independent of the choice of the Hilbert basis $\rho$. We can then assume that the polynomials $\rho_j$ are homogeneous of degree $\a_j$.

On $\R^d$ we define anisotropic dilations by the formula
$$
\delta_r (y_1,\ldots,y_k)=(r^{\a_1}\,y_1,\ldots,r^{\a_d}\,y_d)
\qquad \forall r>0\ ,
$$
and we shall denote by $|\cdot|_{\a}$ a corresponding 
homogeneous gauge, e.g.
\begin{equation}\label{gauge}
|y|_{\a}=c\sum_{j=1}^d |y_j|^{1/\a_j}\ ,
\end{equation}
satisfying $|\delta_r y|_{\a}=  r |y|_{\a}$.

On $\R^m$ we keep isotropic dilations, given by scalar multiplication. 
Clearly, $\rho$ is homogeneous of degree 1 with respect to these dilations, i.e.,
$$
\rho(rx)=\delta_r\bigl(\rho(x)\bigr)
\qquad\forall r>0,\quad x\in \R^m\ ,
$$
and $\Sigma$ is $\delta_r$--invariant for all $r>0$.

Since $\rho$ is continuous,
the image under $\rho$ of the unit sphere in $\R^m$
is a compact set not containing 0 (in fact, since $|x|^2$ is a polynomial in the $\rho_j$, cf. \cite{Ma}, $\rho_j(x)=0$ for every $j$ implies that $x=0$).

Choosing the constant $c$ in (\ref{gauge}) appropriately, we can assume that $1\leq |\rho(x)|_{\a} \leq R$ for every $x$ in the unit sphere in $\R^m$. It follows by homogeneity that for every $a,b$, $0\le a<b$,
\begin{equation}\label{supporto}
\rho\left(\{x: a\leq |x|\leq b\}\right) \subset \{y\,:\, a\leq |y|_{\a}\leq Rb\}.
\end{equation}

Fix $\mathcal E:C^\infty(\R^m)_K\rightarrow C^\infty(\R^d)$ a continuous linear operator satisfying the condition $(\mathcal E f)\circ\rho=f$, whose existence is guaranteed by Mather's theorem. 

Denote by $B_s$ the subset of $\R^d$ where $|y|_{\a}<s$. For every $p\in\N$ there is $q\in N$ such that, for $f$ supported in the unit ball,
$$
\|\mathcal Ef\|_{C^p(B_{R^2})}<C_p\,\|f\|_{C^q}\ .
$$

Given $r>0$, set $f_r(x)=f(rx)$ and
$$
\mathcal E_rf=(\mathcal Ef_r)\circ\delta_{r^{-1}}\ .
$$

By the homogeneity of $\rho$, $(\mathcal E_r f)\circ\rho=f$.
For $r>1$ and $f$ supported on the ball of radius $r$, we have
\begin{equation}\label{stimaE}
\|\mathcal E_rf\|_{C^p(B_{rR^2})}<C_p\, r^q\, \|f\|_{C^q}\ .
\end{equation}

Let $\{\varphi_j\}_{j\ge0}$ be a partition of unity
on $\R^m$ consisting of radial smooth functions such that
\begin{enumerate}
\item[{($a$) }] $\varphi_0$ is supported on $\{x:|x|<1\}$;
\item[{($b$) }] for $j\ge1$, $\varphi_j$ is supported on $\{x:R^{j-2}<|x|<R^j\}$;
\item[{($c$) }] for $j\ge1$, $\varphi_j(x)=\varphi_1(R^{-(j-1)}x)$.
\end{enumerate}

Similarly, let $\{\psi_j\}_{j\ge0}$ be a partition of unity on $\R^d$ consisting of smooth functions such that
\begin{enumerate}
\item[{($a'$) }] $\psi_0$ is supported on $\{y:|y|_{\a}<R\}$;
\item[{($b'$) }] for $j\ge1$, $\psi_j$ is supported on $\{y:R^{j-1}<|y|_{\a}<R^{j+1}\}$;
\item[{($c'$) }] for $j\ge1$, $\psi_j(y)=\psi_1(\delta_{R^{-(j-1)}}y)$.
\end{enumerate}

For $f\in \cS_K(\R^m)$ define

$$
\mathcal E'f(y)=\sum_{j=0}^\infty\sum_{\ell=-2}^1\psi_{j+\ell}(y)\mathcal E_{R^j}(\varphi_jf)=\sum_{j=0}^\infty\psi_j(y)\sum_{\ell=-2}^1\mathcal E_{R^{j-\ell}}(\varphi_{j-\ell}f)\ ,
$$
with the convention that $\psi_{-1}=\psi_{-2}=\varphi_{-1}=0$.
Then
$$
\mathcal E'f\big(\rho(x)\big)=\sum_{j=0}^\infty\sum_{\ell=-2}^1\psi_{j+\ell}\big(\rho(x)\big)\varphi_j(x)f(x)\ .
$$

By (\ref{supporto}), $\sum_{\ell=-2}^1\psi_{j+\ell}\big(\rho(x)\big)=1$ on the support of $\varphi_j$, hence $\mathcal E'f\circ\rho=f$.

We have the following estimate for the Schwartz norms in (\ref{normeSchwartz}):
$$
\|f\|_{(p,\R^m)}\sim \sum_{j=0}^\infty R^{jp}\|f\varphi_j\|_{C^p}\ .
$$

On $\R^d$ we adapt the Schwartz norms to the dilations $\delta_r$ by setting
$$
\|g\|'_{(p,\R^d)}=\sup_{y\in\R^d,\sum a_j\alpha_j\le p}(1+|y|_{\a})^p\big|\partial^\alpha g(y)\big|\ .
$$

We then have
$$
\|g\|'_{(p,\R^d)}\sim\sum_{j=0}^\infty R^{jp}\sup_{\sum a_j\alpha_j\le p}\|\partial^\alpha g\,
\psi_j\|_\infty
\lesssim \sum_{j=0}^\infty R^{jp}\|g\psi_j\|_{C^p}\ .
$$

Therefore
$$
\begin{aligned}
\|\mathcal E'f\|'_{(p,\R^m)}&\le C\sum_{j=0}^\infty R^{jp}\Big\|\psi_j\sum_{\ell=-2}^1\mathcal E_{R^{j-\ell}}(\varphi_{j-\ell}f)\Big\|_{C^p}\\
&\le C\sum_{j=0}^\infty\sum_{\ell=-2}^1 R^{jp}\|\mathcal E_{R^{j-\ell}}(\varphi_{j-\ell}f)\|_{C^p(B_{R^{j+1}})}\\
&=C\sum_{j=0}^\infty\sum_{\ell=-2}^1 R^{jp}\|\mathcal E_{R^j}(\varphi_jf)\|_{C^p(B_{R^{j+\ell+1}})}\\
&=C\sum_{j=0}^\infty R^{jp}\|\mathcal E_{R^j}(\varphi_jf)\|_{C^p(B_{R^{j+2}})}\ .
\end{aligned}
$$

By (\ref{stimaE}), since $\varphi_jf$ is supported on the ball of radius $R^j$,
$$
\|\mathcal E'f\|'_{(p,\R^m)}\le C_p\sum_{j=0}^\infty R^{j(p+q)}\|\varphi_jf\|_{C^q}\le C_p\|f\|_{(p+q,\R^m)}\ .
$$
\end{proof}

\section{Schwartz extensions of the Gelfand transform of $f\in\cS_K(H_n)$}
\label{viceversa}

In this section we suppose that $K$ is a closed connected subgroup
of $\Un$. The following theorem settles the proof of Theorem~\ref{main}
in this case.

\begin{theorem}\label{main2} Let  $f$ be in $\cS_K(\Hn)$. 
For every $p$ in $\N$ there exist $F_p$ in $\cS(\R^{d+1})$ and $q$ in $\N$, both depending on $p$,  such that  
${F_p}_{|_{\Sigma_K^V}}=\widehat f$
and
$
\left\|
F_p
\right\|_{(p,\R^{d+1})}\leq C_p\,\|f\|_{(q,\Hn)}.
$
\end{theorem}

Notice that this statement implies the existence of a continuous map from $\cS(\Sigma_K^V)$ to $\cS_K(H_n)$ that inverts the Gelfand transform, even though its formulation is much weaker than that of Theorem \ref{Schwarz-Mather}. We do not claim that for each $f$ a single $F$ can be found, all of whose Schwartz norms are controlled by those of $f$. In addition, our proof does not show if $F_p$ can be chosen to be linearly dependent on $f$.

The proof of Theorem \ref{main2} 
 is modelled on that given in \cite{ADR} for the cases $K=\Un ,\mathbb T^n$, but with some relevant differences. On one hand we present a simplification of the argument given there, disregarding the partial results concerning extensions of $\widehat f$ with finite orders of regularity; on the other hand extra arguments are required in the general setting.

We need to show that the Gelfand transform $\widehat f$ of $f\in\cS_K(H_n)$ extends from $\Sigma_K^V$ to a Schwartz function on $\R^{d+1}$.
Our starting point is the construction of a Schwartz extension to all of $\{0\}\times\R^d$ of the restriction of $\widehat f$ to 
$$
\Sigma_0=\{\widehat V(\eta_{Kw}):w\in\C^n\}=\{0\}\times\rho(\C^n)\ .
$$

If $\cF f$ denotes the Fourier transform
 in $\C^n\times\R$,
$$
\cF f(\l,w)=\int_{\C^n\times\R}f(t,z)e^{-i(\l t+\RE z\cdot\bar w)}\,dw\,dt\ ,
$$
we denote
\begin{equation*}
\tilde f(w)=\cF f(0,-w)=\hat f\big(0,\rho(w)\big)\ .
\end{equation*}
To begin with, we set
\begin{equation}\label{sharp}
{\widehat{f}\,}^\sharp=\mathcal E'\tilde f\in\cS(\R^d)\ .
\end{equation}

Then 
${\widehat{f}\,}^\sharp(\xi)=\widehat f(0,\xi)$ if $(0,\xi)\in\Sigma_0$.

\medskip

The next step consists in producing a Taylor development of $\widehat f$ at $\l=0$. The following result is derived from~\cite{G}. In our setting the formula must take into account the extended functions in (\ref{sharp}).

\begin{proposition}\label{geller}
Let $f$ be in $\cS_K(\Hn)$. Then
there exist functions $f_j$, $j\geq 1$, in $\cS_K(\Hn)$, depending linearly and
continuously on $f$, such that
for any $p$ in $\N$,
$$
\widehat{f}(\l,\xi)
=\sum_{j=0}^p
\frac{\l^j}{j!}{\widehat{f_j}\,}^\sharp
(\xi)
+
\frac{\l^{p+1}}{(p+1)!}\,{\widehat{f_{p+1}}}
(\l,\xi),
\qquad
\forall(\l,\xi)\in \Sigma_K,
$$
where $f_0=f$ and ${\widehat{f_j}\,}^\sharp$ is obtained from $f_j$ applying {\rm (\ref{sharp})}.
\end{proposition}

\begin{proof}
For $f$ in $\cS_K(\Hn)$, we claim that the restriction of
$u(\l,\xi)={\widehat{f}\,}^\sharp (\xi)$ to $\Sigma_K^V$ is in $\cS(\Sigma_K^V)$. It is quite obvious that $u$ is smooth. Let $\psi$ be a smooth function on the line, equal to 1 on $[-2,2]$ and supported on $[-3,3]$. Define
\begin{equation*}
\begin{aligned}
\Psi(\l,\xi)&=\psi(\l^2+\xi_1^{2/m_1}+\cdots+\xi_d^{2/m_d})\\
&\qquad +\psi\bigg(\frac{\l^2}{\xi_1^{2/m_1}+\cdots+\xi_d^{2/m_d}}\bigg)\big(1-\psi(\l^2+\xi_1^{2/m_1}+\cdots+\xi_d^{2/m_d})\big)\ .
\end{aligned}
\end{equation*}

By (\ref{scaling}) and (\ref{due}) in Lemma \ref{autinteri}, $\Psi$ is equal to 1 on a neighborhood of $\Sigma_K^V$. It is also homogeneous of degree 0 with respect to the dilations (\ref{dilations}) outside of a compact set. Then $\Psi u$ is in $\cS(\R^{d+1})$ and coincides with $u$ on $\Sigma_K^V$.

It follows from Corollary ~\ref{corveneruso}  that there exists $h$
in $\cS_K(\Hn)$ such that
$$
\widehat{f}(\l,\xi)
-{\widehat{f}\,}^\sharp (\xi)
=\widehat{h}(\l,\xi)
\qquad \forall (\l,\xi)\in \Sigma^V_K.
$$

Since $\widehat h\big(0,\rho(w)\big)=0$ for every $w$,
$\int_{-\infty}^{+\infty}h(z,t)\,dt=0$ for every $z$.
Therefore
$$
f_1(z,t)=\int_{-\infty}^{t}h(z,s)\,ds
$$
is in $\cS_K(\Hn)$ and
$$
\widehat h(\l,\xi)=\l\,\widehat f_1(\l,\xi)
\qquad\forall (\l,\xi)\in\Sigma^V_K.
$$

It easy to verify that the map $U:f\mapsto  f_1$ 
is linear and continuous on $\cS_K(\Hn)$.
We then define $f_j$, $j\geq 1$, by the recursion
formula $f_j=jUf_{j-1}$
and the thesis follows by induction.
\end{proof}

We use now the Whitney Extension Theorem \cite{Malgrange} to extend the $C^\infty$-jet $\{\partial_\xi^\a{\widehat{f_j}\,}^\sharp\}_{(j,\alpha)\in\N^{d+1}}$ to a Schwartz function on $\R^{d+1}$. In doing so, we must keep accurate control of the Schwartz norms. For this purpose we use Lemma 4.1 in \cite{ADR}, which reads as follows.

\begin{lemma}\label{lemma4.1}
Let $k\ge1$ and let $h(\l,\xi)$ be a $C^k$-function on $\R^m\times\R^n$ such that
\begin{enumerate}
\item $\partial_\l^\a h(0,\xi)=0$ for $|\a|\le k$ and $\xi\in\R^n$;
\item for every $p\in\N$,
$$
\a_p(h)=\sup_{|\a|+|\beta|\le k}\big\|(1+|\cdot|)^p\partial_\l^\a\partial_\xi^\beta h\big\|_\infty<\infty\ .
$$
\end{enumerate}

Then, for every $\eps>0$ and $M\in\N$, there exists a function $h_{\eps,M}\in\cS(\R^m\times\R^n)$ such that
\begin{enumerate}
\item $\partial_\l^\a h_{\eps,M}(0,\xi)=0$ for every $\a\in\N^m$ and $\xi\in\R^n$;
\item $\sup_{|\a|+|\beta|\le k-1}\big\|(1+|\cdot|)^M\partial_\l^\a\partial_\xi^\beta (h-h_{\eps,M})\big\|_\infty<\eps$;
\item for every $p\in\N$ there is a constant $C_{k,p,M}$ such that
$$
\|h_{\eps,M}\|_{(p,\R^{m+n})} \le C_{k,p,M}\big(1+\a_M(h)^p\eps^{-p}\big)\big\|(1+|\cdot|)^ph\big\|_\infty\ .
$$
\end{enumerate}
\end{lemma}

The following proposition is in \cite[Proposition 4.2]{ADR} for $K=\mathbb T^n$. We give here a simplified proof.

\begin{proposition}\label{whitney}
Given $f\in\cS_K(H_n)$ and $p\in\N$, there are $H\in \cS(\R^{d+1})$ and $q\in\N$ such that $\partial_\l^jH(0,\xi) ={\widehat{f_j}\,}^\sharp(\xi)$ and
$$
\|H\|_{(p,\R^{d+1})}\le C_p\|f\|_{(q,H_n)}\ .
$$
\end{proposition}

\begin{proof}
Let $\eta$ be a smooth function on $\R$ such that
$\eta(t)=1$ if $|t|\leq 1$ and $\eta(t)=0$ if $|t|\geq 2$. By Theorem \ref{Schwarz-Mather} and  Proposition \ref{geller}, for every $k$ and $r$ there exists $q_{k,r}$ such that
\begin{equation}\label{stimasharp}
\|\widehat{f_k}^\sharp\|_{(r,\R^{d})}\le C_{k,r}\|f\|_{(q_{k,r},H_n)}\ .
\end{equation}

We fix $p\in\N$ and apply Lemma \ref{lemma4.1} to 
$$
h_k(\l,\xi)=\eta(\l)\frac{\l^{k+1}}{(k+1)!}\widehat{f_{k+1}}^\sharp(\xi)\ .
$$ 

Hypothesis (1) is obviously satisfied and (2) also, because $h_k$ is a Schwartz function. By (\ref{stimasharp}),
\begin{equation}\label{stimaalphak}
\a_r(h_k)\le C_{k,r}\|f\|_{(q_{k+1,r},H_n)}\ .
\end{equation}

Let  $q$ be the maximum among  the $q_{k,p}$ for $k\le p+1$.  Setting  $\eps_k=2^{-k}\|f\|_{(q,H_n)}$, $M=p$, for each $k$ there is a function $H_k\in\cS(\R^{d+1})$ such that
\begin{enumerate}
\item[(i)]  $\partial_\l^j H_k(0,\xi)=0$ for all $j\in\N$ and $\xi\in\R^d$;
\item[(ii)] $\sup_{|\a|+|\beta|\le k-1}\big\|(1+|\cdot|)^p\partial_\l^\a\partial_\xi^\beta (h_k-H_k)\big\|_\infty<2^{-k}\|f\|_{(q,H_n)}$;
\item[(iii)] for $k\le p$, using (\ref{stimaalphak}),
$$
\|H_k\|_{(p,\R^{d+1})}\le  C_{k,p}\big(1+\eps_k^{-p}\|f\|_{(q,H_n)}^p\big)\big\|(1+|\cdot|)^ph_k\big\|_\infty
\le C_p\|f\|_{(q,H_n)}\ .
$$
\end{enumerate}

Define 
$$
H=\sum_{k=0}^ph_k-\sum_{k=0}^pH_k+\sum_{k=p+1}^\infty(h_k-H_k)\ .
$$

By (\ref{stimasharp}), (ii) and (iii), the $p$-th Schwartz norm of $H$ is finite and controlled by a constant times the $q$-th Schwartz norm of $f$. Differentiating term by term, using (i) and the identity $\partial_\l^jh_k(0,\xi)=\delta_{j,k+1}\widehat{f_{k+1}}^\sharp(\xi)$, we obtain that $\partial_\l^jH(0,\xi) =\widehat{f_j}^\sharp(\xi)$ for every $j$.
\end{proof}

\bigskip

Let now $\vp$ be a smooth function on $\R$ such that
$\vp(t)=1$ if $|t|\leq 1/2$ and $\vp(t)=0$ if $|t|\geq 3/4$.
For $h$ defined on $\Sigma_K^V$,  we define the  function $Eh$ on  $\R^{d+1}$ by

\begin{equation*}
Eh(\l,\xi)=
\begin{cases}
\displaystyle
\sum_{\a\in\Lambda}
h\big(\l,\xi_{(\l,\a)}\big)
\prod_{\ell=1}^{d}\vp
\left( \frac{\xi_{\ell}}{|\l|^{m_\ell}}
-\widehat{V_\ell}(\phi_\a)\right)
&\l\neq 0,\,\,  \xi\in \R^d
\\
\quad 0
&\l=0,\,\, \xi\in \R^d,
\end{cases}
\end{equation*}
where
$
\xi_{(\l,\a)} =\big(
\widehat{V_1}(\phi_{\l,\alpha}),
\ldots,\widehat{V_d}(\phi_{\l,\alpha})
\big)
=\big(
|\l|^{m_1}\widehat{V_1}(\phi_{\alpha}),
\ldots,|\l|^{m_d}\widehat{V_d}(\phi_{\alpha})
\big)
$.

Recall that, by Lemma~\ref{autinteri}, each $\widehat{V_\ell}(\phi_\a)$
is a  positive integer. Therefore, if $\xi_\ell\le0$ for some $\ell$,
every term in the series vanishes, whereas, if
 $\xi$ is in $\R^d_+$,  the series reduces to
at most one single term.
Moreover,  for every $g$ in $\cS(\Hn)$,
 $E\widehat g=\widehat g$ on $\Sigma'$.

The proof of the following result goes as for~\cite[Lemma 3.1]{ADR}, using \cite[p. 407]{BJR} instead of~{\cite[(2.2)]{ADR}}. In contrast with~\cite{ADR} we state it only for vanishing of infinite order of the Taylor development of the Gelfand transform on $\Sigma_0$.

\begin{proposition}\label{estensioneCk}
Suppose that
$g$  in $\cS_K(\Hn)$ and
$
\widehat{g_j}_{|_{\Sigma_0}}
=0$ for every $j$.
Then 
\begin{enumerate}
\item 
$\displaystyle
E \widehat {g}(\l,\xi)=\widehat g(\l,\xi)$ for all $(\l,\xi)\in\Sigma^V_K$;
\item 
 $ \partial_\l^s (E\widehat g)(0,\xi)=
0$ for all $s$ and $\xi\in \R^d$;
\item  
for every $p\geq 0$ there exist a constant $C_p$ and an integer $q\geq 0$
such that
$$
\|E\widehat g\|_{(p,\R^{d+1})}
\leq C_{p}\|g\|_{(q,\Hn)}\ .
$$
\end{enumerate}

In particular, $E \widehat{g}\in\cS(\R^{d+1})$.
\end{proposition}

\bigskip

To conclude the proof of Theorem \ref{main2}, take  $f$ in $\cS_K(\Hn)$
 and $p$ in $\N$.
 Let  $H$ be the function in  $\cS(\R^{d+1})$,
  depending on $p$,  defined as in  
  Proposition~\ref{whitney}.
By~Theorem~\ref{Veneruso}, there exists $h$ in  $\cS(\Hn)$
such that
$$
H_{|_{\Sigma^V_K}}=\widehat h\ .
$$ 

Define
$$
F=E(\widehat f-\widehat  h)+H\ ,
$$
and the thesis follows easily.

\section{General compact groups of automorphisms of $\Hn$}
\label{Kgenerico}

We have discussed in the previous sections the Gelfand pairs 
associated with connected
subgroups of $\Un$.
In this section we only assume that $K$
is a compact group of automorphisms of $\Hn$.
Let $K_0$ be the connected identity component of $K$. 
Then $K_0$ is a normal subgroup of $K$ and $F=K/K_0$ is a finite group.
Conjugating $K$ with an automorphism if necessary,
we may suppose that $K_0$ is a subgroup of $\Un$.

For $D$ in $ {\mathbb D}_{K_0}$ and $w=kK_0$ in $F$, define $D^w$ by
$$
D^w f=D(f\circ k^{-1})\circ k\ \qquad \forall f\in C^\infty(H_n).
$$
Since $K_0$ is normal,  $D^w$ is in ${\mathbb D}_{K_0}$. 
It is also clear that ${\mathbb D}_{K_0}$ admits a generating set 
which is stable under the action of the group $F$. 
Indeed, it suffices to add to any given system 
of generators the $F$-images of its elements. 
Denoting by $\cV$ the linear span of these generators, 
$F$ acts linearly on $\cV$. Let $V=\{V_1,\dots,V_d\}$ be a basis of $\cV$, 
orthonormal with respect to an $F$--invariant scalar product. 
Clearly, $V$ is a generating set for ${\mathbb D}_{K_0}$. 

Applying Hilbert's Basis Theorem as in Section~\ref{Mather}, 
there exists a finite number of (homogeneous) polynomials $\rho_1,\ldots,\rho_r$
generating the subalgebra ${ P}_F(\R^d)$ of $F$--invariant
elements in ${P}(\R^d)$. 
Let $\rho=(\rho_1,\ldots,\rho_r):\R^d\to\R^r$ be the corresponding Hilbert map
and let $W_j=\rho_j(V_1,\ldots,V_d)$ for $j=1,\ldots,r$.

When $f$ is $K_0$--invariant and $w=kK_0$ in $F$, we set $f\circ w=f\circ k$.



\begin{lemma}\label{invar}
The set $W=\set{W_1,
,\ldots,W_r}$  generates
${{\mathbb D}_K}$. Moreover
if $\psi$ is a  $K$-spherical function, then
\begin{equation}\label{formsferiche}
\psi=\frac1{|F|} \sum_{w\in F} \phi\circ w,
\end{equation}
for some $K_0$-spherical function $\phi$.
\end{lemma}

\begin{proof}
Take $D$ in ${{\mathbb D}_K}$.  As an element of ${{\mathbb D}_{K_0}}$,
 $D$ is a polynomial in the $V_j$. Averaging over the action of $F$, 
 we can express $D$ as an $F$--invariant polynomial in the $V_j$. Hence  
 $D$ is a polynomial in $W_1=\rho_1(V_1,\ldots,V_d),\ldots,W_r=\rho_r(V_1,\ldots,V_d)$.

Recall that all $K_0$- and $K$-spherical functions are of positive type~\cite{BJR90}.
Let $\cP_K$ (resp.~$\cP_{K_0}$) denote the convex set of 
 $K$--invariant (resp.~$K_0$--invariant) functions of positive type
 equal to $1$ 
 at the identity element, and consider the linear map 
 $J:L^\infty_{K_0}\rightarrow L^\infty_K$ defined by
 $J\varphi=\frac1{|F|}\sum_{w\in F}\varphi\circ w$.
Since $\cP_{K_0}$ and $\cP_K$ are weak*-compact and $J$ maps $\cP_{K_0}$ to $\cP_K$, the extremal points of $\cP_K$ are images of extremal points of $\cP_{K_0}$. This proves that every $K$-spherical function has the form (\ref{formsferiche}). 

Conversely, if $\psi$ is given by (\ref{formsferiche}) and $D\in{\mathbb D}_K$, then
$$
D\psi=\frac1{|F|}\sum_{w\in F}D(\varphi\circ w)=
\frac1{|F|}\sum_{w\in F}(D\varphi)\circ w=\hat D(\varphi)\, \psi\ ,
$$
showing that $\psi$ is $K$-spherical.
\end{proof}

From Lemma~\ref{invar} we derive
the following property of the Gelfand spectra:
$$
\rho(\Sigma_{K_0}^V)=\Sigma_K^{W}\subset \R^r.
$$
If $V$ is as above, the linear action of $F$ on $\cV$ leaves $\Sigma^V_{K_0}$ invariant. 
For a $K_0$--invariant function $f$ and $w$ in $F$, 
\begin{equation}
\label{numero}
 \cG_V(f\circ w)=(\cG_Vf)\circ w 
\end{equation}
Let $\cS_F(\Sigma_{K_0}^V)$ be 
the space of $F$--invariant elements in $\cS(\Sigma_{K_0}^V)$.

\begin{lemma}\label{schspettri}
The map $f\mapsto f\circ\rho$ is an isomorphism 
between $\cS(\Sigma_{K}^{W})$ and 
$\cS_F(\Sigma_{K_0}^V)$.
\end{lemma}

\begin{proof} If $f$ is in $\cS(\Sigma_{K}^{W})$, let $\tilde f$ 
be any Schwartz extension of $f$ to $\R^{r}$. 
Then $g=\tilde f\circ\rho$ is an $F$--invariant Schwartz function on $\R^d$ 
and its restriction to $\Sigma_{K_0}^V$ is  $f\circ\rho$. 
This proves the continuity of the map.

Conversely, given $g$ in $\cS_F(\Sigma_{K_0}^V)$, let $\tilde g$ be an 
$F$--invariant Schwartz extension of $g$ to $\R^d$.
Set 
$h=(\cE'\tilde g)_{|_{\Sigma_{K}^{W}}}$, where $\cE'$ 
is the operator of Theorem~\ref{Schwarz-Mather}
for the group $F$.
The proof that the dependence of $h$ on $g$ is continuous 
is based on the simple observation that, 
for any Schwartz norm $\|\ \|_{(N)}$ on $\R^d$, 
the infimum of the norms of all extensions of $g$ is the same 
as the infimum restricted to its $F$--invariant extensions. 
\end{proof}

We can now prove  Theorem~\ref{main} for general $K$. 
Assume that $K$ is a compact group of automorphisms of $\Hn$ and let $K_0$, $F$, $V_1,\dots, V_d$, $\rho$ be as above. 

Take $f$ in $L^1_K(H_n)$. Denote by $\cG_Vf$ (resp. $\cG_{W}f$) 
its Gelfand tranform as a $K_0$--invariant (resp. $K$--invariant) function. 
Then $\cG_Vf=\cG_{W}f\circ\rho$. In particular, 
a $K_0$--invariant function is $K$--invariant if and only if 
$\cG_Vf$ is $F$--invariant.

If $f$ is in $\cS_K(H_n)$, then $f$ is also $K_0$--invariant and
$\cG_Vf$ is in $\cS_F(\Sigma_{K_0}^V)$ by Theorem~\ref{main2}; 
therefore $\cG_{W}f$ is in $\cS(\Sigma_{K}^{W})$
 by Lemma~\ref{schspettri}.

Conversely, if $\cG_{W}f$ is in $\cS(\Sigma_{K}^{W})$, 
it follows as before that 
$\cG_Vf$ is in $\cS_F(\Sigma_{K_0}^V)$ and therefore $f$ is in $\cS_K(H_n)$
by~(\ref{numero}).


\end{document}